\documentclass[12pt]{article}

\usepackage{tikz}

\usepackage{url}
\usepackage{graphicx}
\usepackage{comment}
\usepackage{mathrsfs}
\usepackage{kpfonts}
\usepackage{nameref}
\usepackage[shortlabels,inline]{enumitem}
\usepackage{amsmath}
\usepackage[nobysame]{amsrefs}
\usepackage{amssymb}
\usepackage{empheq}
\usepackage[sort]{cite}
\usepackage[shortlabels]{enumitem}
\setlist[enumerate]{nosep}
\usepackage{abstract}
\usepackage[doc]{optional}
\usepackage{xcolor}
\definecolor{lightgray}{gray}{0.9}
\usepackage[colorlinks=true,
linkcolor=refkey,
urlcolor=lblue,
citecolor=red]{hyperref}
\usepackage{float}
\usepackage{soul}

\definecolor{labelkey}{rgb}{0,0.08,0.45}
\definecolor{refkey}{rgb}{0,0.6,0.0}
\definecolor{Brown}{rgb}{0.45,0.0,0.05}
\definecolor{lime}{rgb}{0.00,0.8,0.0}
\definecolor{lblue}{rgb}{0.5,0.5,0.99}

\usepackage{mathpazo}

\usepackage{subcaption}
\colorlet{hlcyan}{cyan!30}

\usepackage{stmaryrd}
\makeatletter
\def\namedlabel#1#2{\begingroup
	\def\@currentlabel{#2}%
	\label{#1}\endgroup
}
\makeatother

\newcommand{\menge}[2]{\big\{{#1}~\big |~{#2}\big\}}
\newcommand{\mmenge}[2]{\bigg\{{#1}~\bigg |~{#2}\bigg\}}

\newcommand{\fenv}[1]%
{\ensuremath{\,\overrightarrow{\operatorname{env}}_{#1}}}
\newcommand{\benv}[1]%
{\ensuremath{\,\overleftarrow{\operatorname{env}}_{#1}}}

\newcommand{\scal}[2]{\left\langle{#1},{#2}  \right\rangle}

\newcommand{\RR}{\ensuremath{\mathbb R}}

\newcommand{\Ct}{\ensuremath{\widetilde{C}}}

\newcommand{\argmin}{\ensuremath{\operatorname{argmin}}}

\newcommand{\Id}{\ensuremath{\operatorname{Id}}}

\newcommand{\proj}[1]{{\thinspace P\thinspace}_%
	{\negthinspace\negthinspace #1}}

\newcommand{\graphc}{\stackrel{g}{\rightarrow}}

\DeclarePairedDelimiterX\set[2]{ \{ }{ \}_{#2} }{#1}

\DeclarePairedDelimiterX\rb[1]{ ( }{ ) }{#1}

{\begin{list}{}{%
			\settowidth{\labelwidth}{\textrm{#1~}}%
			\setlength{\leftmargin}{\labelwidth+\labelsep}}}%requires macro calc.sty
	{\end{list}}
\usepackage{amsthm}
\usepackage[capitalize,nameinlink]{cleveref}
\crefname{equation}{}{equations}
\crefname{chapter}{Appendix}{chapters}
\crefname{item}{}{items}
\crefname{enumi}{}{}

\theoremstyle{definition}
\newtheorem{theorem}{Theorem}[section]

	\newtheorem{proposition}[theorem]{Proposition}
	
	\newtheorem{definition}[theorem]{Definition}
	
	\newtheorem{example}[theorem]{Example}
	\newtheorem{ex}[theorem]{Example}
	\newtheorem{fact}[theorem]{Fact}
	\newtheorem{remark}[theorem]{Remark}

	\usepackage{multirow}

	\providecommand{\norm}[1]{\lVert#1\rVert}

	\providecommand{\lam}{\lambda}
	\providecommand{\RR}{\mathbb{R}}
	\providecommand{\proj}{\operatorname{Proj}}
	
	\providecommand{\gr}{\operatorname{gra}}
	
	\providecommand{\Id}{\operatorname{{ Id}}}

	\providecommand{\argmin}{\mathrm{arg}\!\min}

	\providecommand{\gr}{\operatorname{gra}}

	\providecommand{\Id}{\operatorname{Id}}

	\providecommand{\RR}{\mathbb{R}}

	\definecolor{myblue}{rgb}{.8, .8, 1}
	\newcommand*\mybluebox[1]{%
		\colorbox{myblue}{\hspace{1em}#1\hspace{1em}}}

	\allowdisplaybreaks % or locally if problems {\allowdisplaybreaks
		\usepackage{blkarray}
		\usepackage{multirow}

		\newcommand{\Calp}{{C_{\alpha}}}
		\newcommand{\tCalp}{{\widetilde{C}_{\alpha}}}

		\usepackage{amsmath}
		
		\usetikzlibrary{calc,intersections}
		\usepackage{pgfplots}
		\usepackage[customcolors]{hf-tikz}
		\tikzset{style green/.style={
				set fill color=green!50!lime!60,
				set border color=white,
			},
			style cyan/.style={
				set fill color=cyan!90!blue!60,
				set border color=white,
			},
			style gray/.style={
				set fill
				color=black!10!,
				set border color=white,
			},
			style orange/.style={
				set fill color=orange!80!red!60,
				set border color=white,
			},
			hor/.style={
				above left offset={-0.15,0.31},
				below right offset={0.15,-0.125},
				#1
			},
			ver/.style={
				above left offset={-0.1,0.3},
				below right offset={0.15,-0.15},
				#1
			}
		}
		
		\pgfplotsset{compat=1.16}
		
\begin{document}
			
			\title{\textsc{
					Projecting onto rectangular hyperbolic paraboloids in Hilbert space
			}}
			
			\author{
				Heinz H.\ Bauschke\thanks{
					Department of Mathematics, University
					of British Columbia,
					Kelowna, B.C. V1V~1V7, Canada. E-mail:
					\texttt{heinz.bauschke@ubc.ca}.},~
				Manish Krishan Lal\thanks{
					Department of Mathematics, University
					of British Columbia,
					Kelowna, B.C. V1V~1V7, Canada. E-mail:
					\texttt{manish.krishanlal@ubc.ca}.},~
				and
				Xianfu Wang\thanks{
					Department of Mathematics, University
					of British Columbia,
					Kelowna, B.C. V1V~1V7, Canada. E-mail:
					\texttt{shawn.wang@ubc.ca}.}
			}
			\date{June 9, 2022} 
			\maketitle
			
			% \vskip 8mm
			
			\begin{abstract} 
				In $\RR^3$, a hyperbolic paraboloid is a classical saddle-shaped quadric surface. Recently, Elser has modeled problems arising in Deep Learning using rectangular hyperbolic paraboloids in $\RR^n$. Motivated by his work, we provide a rigorous analysis of the associated projection. In some cases, 
				finding this projection amounts to finding a certain root of a quintic or cubic polynomial. 
				We also observe when the projection is not a singleton and point out connections to graphical and set convergence.
			\end{abstract}
			
			{
				\noindent
				{\bfseries 2020 Mathematics Subject Classification:}
				{Primary 41A50, 90C26; Secondary 15A63, 46C05. 
				}
				
				\noindent {\bfseries Keywords:}
				cross, graphical convergence, rectangular hyperbolic paraboloid, projection onto a nonconvex set. 
			}
			
			\section{Introduction}
			Throughout this paper, we assume that
			\begin{empheq}[box=\mybluebox]{equation*}
				\text{$X$ is
					a real Hilbert space with an inner product
					$\scal{\cdot}{\cdot}\colon X\times X\to\RR$, }
			\end{empheq}
			and induced norm $\|\cdot\|$, and that 
			$\alpha\in \RR\smallsetminus\{0\}$ and $\beta>0$. 
			Define the $\beta$-weighted norm on
			the product space $X\times X\times\RR$ by
			$$(\forall (x,y, \gamma)\in X\times X\times \RR)\ \|(x,y,\gamma)\| :=\sqrt{\|x\|^2+\|y\|^2+\beta^2|\gamma|^2}.$$
			Now define the set 
			\begin{empheq}[box=\mybluebox]{equation}
				\Calp:= \menge{(x,y,\gamma)\in X\times X\times \RR}{\scal{x}{y}=\alpha \gamma}.
			\end{empheq}
			The set $C_\alpha$ is a special bilinear constraint set in optimization,
			and it corresponds to a rectangular (a.k.a.\ orthogonal) 
			hyperbolic paraboloid in geometry \cite{odehnal2020universe}. 
			Motivated by Deep Learning, Elser recently presented in \cite{elser2021learning} a formula for 
			the projection $\proj{\Calp}(x_{0}, y_{0},\gamma_{0})$ when $x_{0}\neq \pm y_{0}$. However, complete mathematical justifications were not presented, and the case when $x_0=\pm y_0$ was not considered. 
			The goal of this paper is to provide a complete analysis of $\proj{\Calp}$ that is applicable to all possible cases. 
			
			The paper is organized as follows. 
			We collect auxiliary results in 
			\cref{sec:aux}. Our main result is proved in \cref{sec:main} which also
			contains a numerical illustration. 
			The formula for the projection onto the set $\Calp$ is presented in 
			\cref{sec:further}.

			As usual, the distance function and projection mapping associated to
			$\Calp$ are denoted by
			$d_{\Calp}(x_{0},y_{0},\gamma_{0}):=\inf_{(x,y,\gamma)\in\Calp}\|(x,y,\gamma)-(x_{0},y_{0},\gamma_{0})\|$
			and
			$\proj{\Calp}(x_{0},y_{0},\gamma_{0}):=
			\argmin_{(x,y,\gamma)\in \Calp}\|(x,y,\gamma)-(x_{0},y_{0},\gamma_{0})\|$,
			respectively. We say that $x,x_{0}\in X$ are 
			\emph{conically dependent} if there exists $s\geq 0$ such that
			$x=sx_{0}$ or $x_{0}=sx$.

			\section{Auxiliary results}
			
			\label{sec:aux}
			We start with some elementary properties of $\Calp$, and justify the
			existence of projections onto these sets.
			
			\begin{proposition} The following hold:
				\begin{enumerate}
					\item\label{i:norm-closed}
					The set 
					$\Calp$ is closed.
					If $X$ is infinite-dimensional, then $\Calp$ 
					is not weakly closed; in fact, $\overline{\Calp}^{\text{\rm weak}}=X\times X\times\RR$.
					\item\label{i:proxregular}
					$\Calp$ is prox-regular in $X\times X\times \RR$. Hence, for every point in $(x_{0},y_{0},\gamma_{0})\in \Calp$,
					there exists a neighborhood such that the projection mapping $\proj{\Calp}$ is single-valued.
				\end{enumerate}
			\end{proposition}
			\begin{proof}
				\cref{i:norm-closed}:
				Clearly, $\Calp$ is closed. 
				Thus assume that $X$ is infinite-dimensional. 
				By \cite[Proposition~2.1]{bauschke2023projections}, for every $\gamma\in \RR$, $\overline{\{(x,y)\in X\times X |\ \scal{x}{y}=\alpha\gamma\}}^{\text{\rm weak}}=X\times X$. Thus,
				\begin{align*}
					X\times X\times \RR &=  \bigcup_{\gamma\in\RR}\big(\overline{\{(x,y)\in X\times X |\ \scal{x}{y}=\alpha\gamma\}}^{\text{\rm weak}}\times \{\gamma\}\big)\\
					& \subseteq \overline{\{(x,y,\gamma)\in X\times X\times \RR |\ \scal{x}{y}=\alpha\gamma\}}^{\text{\rm weak}}\subseteq X\times X\times \RR.
				\end{align*}
				
				\cref{i:proxregular}: 
				Set $F\colon X\times X\times\RR\to\RR\colon (x,y,\gamma)\mapsto
				\scal{x}{y}-\alpha \gamma$. 
				Then  $\Calp = F^{-1}(0)$ and $\nabla F(x,y,\gamma)= (y,x,-\alpha) \neq (0,0,0)$ because $\alpha\neq 0$. 
				The prox-regularity of $\Calp$ now follows from \cite[Example~6.8]{rockafellar_variational_1998} when $X=\RR^n$
				or from  \cite[Proposition~2.4]{bernard2005prox} in the general case. 
				Finally, the single-valuedness of the projection locally around every point 
				in $\Calp$ follows from \cite[Proposition~4.4]{bernard2005prox}.
			\end{proof}

			To study the projection onto $\Calp$, it is convenient to introduce
			\begin{empheq}[box=\mybluebox]{equation}
				\tCalp := \menge{(u,v,\gamma)\in X \times X \times \RR}{\norm{u}^2-\norm{v}^2=2\alpha\gamma},
			\end{empheq}
			which is the standard form of a rectangular hyperbolic paraboloid. 
			Define a linear operator $A:X\times X \times \RR\rightarrow X\times X \times \RR$ by sending $(u,v,\gamma)$ to $(x,y,\gamma)$, where 
			$$x=\frac{u-v}{\sqrt{2}}\quad\text{and}\quad y=\frac{u+v}{\sqrt{2}}.$$
			In terms of block matrix notation, we have 
			$$\begin{bmatrix}
				x \\
				y\\
				\gamma \end{bmatrix}=\left[ {\begin{array}{ccc}
					\frac{1}{\sqrt{2}}\Id & -\frac{1}{\sqrt{2}}\Id & 0\\
					\frac{1}{\sqrt{2}}\Id & \frac{1}{\sqrt{2}}\Id & 0\\
					0 & 0 & 1
			\end{array} } \right]  \begin{bmatrix}
				u \\
				v \\
				\gamma
			\end{bmatrix}
			\Leftrightarrow
			\begin{bmatrix}
				u \\
				v\\
				\gamma \end{bmatrix}=\left[ {\begin{array}{ccc}
					\frac{1}{\sqrt{2}}\Id & \frac{1}{\sqrt{2}}\Id & 0\\
					-\frac{1}{\sqrt{2}}\Id & \frac{1}{\sqrt{2}}\Id & 0\\
					0 & 0 & 1
			\end{array} } \right] \begin{bmatrix}
				x \\
				y\\
				\gamma
			\end{bmatrix}.$$
			Thus, we may and do identify $A$ with its block matrix representation 
			$$A= \begin{bmatrix}
				\frac{1}{\sqrt{2}}\Id & -\frac{1}{\sqrt{2}}\Id & 0\\
				\frac{1}{\sqrt{2}}\Id & \frac{1}{\sqrt{2}}\Id & 0\\
				0 & 0 & 1
			\end{bmatrix},$$
			and we denote the adjoint of $A$ by $A^{\intercal}$. 
			Note that $A$ corresponds to a rotation by $\pi/4$ about
			the $\gamma$-axis. 
			The relationship between $\Calp$ and $\tCalp$ is summarized as follows.
			
			\begin{proposition}\label{p:transform}
				The following hold:
				\begin{enumerate}
					\item \label{i:map} $A$ is a 
					surjective isometry (i.e., a unitary operator): $AA^{\intercal}=A^{\intercal}A=\Id$.
					\item \label{i:set} $A\tCalp=\Calp$ and $\tCalp=A^{\intercal}\Calp$.
					\item\label{i:proj}
					$\proj{\Calp}=A\proj{\tCalp}A^{\intercal}.$
				\end{enumerate}
			\end{proposition}
			\begin{proof} 
				It is straightforward to verify \cref{i:map} and \cref{i:set}. 
				To show \cref{i:proj}, let $(x_{0},y_{0},\gamma_{0})\in X\times X\times\RR$.
				In view of \cref{i:map} and \cref{i:set},
				we have
				$(x,y,\gamma)\in \proj{\Calp}(x_{0},y_{0},\gamma_{0})$ if and
				only if $(x,y,\gamma)\in \Calp$ and
				\begin{equation*}
					\|(x,y,\gamma)-(x_{0},y_{0},\gamma_{0})\|=
					d_{\Calp}(x_{0},y_{0},\gamma_{0})=d_{A\tCalp}(x_{0},y_{0},\gamma_{0})=
					d_{\tCalp}(A^{\intercal}[x_{0},y_{0},\gamma_{0}]^\intercal), 
				\end{equation*}
				and this is equivalent to 
				$$\|A^{\intercal}[x,y,\gamma]^\intercal-A^{\intercal}[x_{0},y_{0},\gamma_{0}]^\intercal\|=
				d_{\tCalp}(A^{\intercal}[x_{0},y_{0},\gamma_{0}]^\intercal).$$
				Since $A^{\intercal}[x,y,\gamma]^\intercal\in\tCalp$, this gives
				$A^{\intercal}[x,y,\gamma]^\intercal\in \proj{\tCalp}(A^{\intercal}[x_{0},y_{0},\gamma_{0}]^\intercal)$, i.e.,
				$[x,y,\gamma]^\intercal\in A\proj{\tCalp}(A^{\intercal}[x_{0},y_{0},\gamma_{0}]^\intercal)$.
				The converse inclusion is proved similarly. 
			\end{proof}
			
			Exploiting the structure of $\tCalp$ is crucial for showing the existence
			of $\proj{\tCalp}(u_{0},v_{0},\gamma_{0})$ for every
			$(u_{0},v_{0},\gamma_{0})\in X\times X\times \RR$.
			
			\begin{proposition} {\bf (existence of the projection)} 
				\label{p:three:dim}
				Let 
				$(u_0,v_0,\gamma_{0}) \in X\times X\times \RR$. 
				Then  the minimization problem
				\begin{subequations}
					\label{e:simple}
					\begin{align}
						\textrm{minimize} \quad  f(u,v,\gamma) &:=\|u-u_{0}\|^2+\|v-
						v_{0}\|^2 +\beta^{2}|\gamma-\gamma_{0}|^2 \label{e:simple1}\\
						\textrm{subject to} \quad  \  h(u,v,\gamma) &:=\norm{u}^2 -\norm{v}^2-2\alpha\gamma =0
						\label{e:simple2}
					\end{align}
				\end{subequations}
				always has a solution, i.e., $\proj{\tCalp}(u_{0},v_{0},\gamma_{0})\neq\varnothing$.
				If $(u,v,\gamma)\in \proj{\tCalp}(u_{0},v_{0},\gamma_{0})$, then 
				$u,u_{0}$ are conically dependent, 
				and $v,v_{0}$ are also conically dependent. 
			\end{proposition}
			
			\begin{proof} We only illustrate the case 
				when $u_{0}\neq 0, v_{0}\neq 0$, since
				the other cases are similar. 
				We claim that the optimization problem is essentially
				$3$-dimensional. To this end, we expand
				\begin{align}
					\label{e:220603b}
					f(u,v,\gamma) &
					=\underbrace{\|u\|^2-2\scal{u}{u_{0}}+\|u_{0}\|^2}+\underbrace{\|v\|^2-2\scal{v}{v_{0}}+
						\|v_{0}\|^2}+\beta^2|\gamma-\gamma_{0}|^2.
				\end{align}
				The constraint
				$$h(u,v,\gamma)=\norm{u}^2 -\norm{v}^2-2\alpha \gamma =0$$
				means that for the  variables $u,v$ 
				only the norms $\|u\|$ and $\|v\|$ matter.
				With $\|u\|$ fixed, the Cauchy-Schwarz inequality in Hilbert space 
				(see, e.g., \cite{kreyszig1991introductory}),
				shows that $-2\scal{u}{u_{0}}$ in the left underbraced part of 
				\cref{e:220603b}
				will be smallest when 
				$u,u_{0}$ are conically dependent. 
				Similarly, for fixed $\|v\|$, the second underlined part in
				$f$ will be smaller when
				$v=t v_{0}$ for some $t\geq 0$. It follows that the optimization
				problem given by  \cref{e:simple}  is equivalent to
				\begin{subequations}
					\label{e:sunday}
					\begin{align}
						\textrm{minimize}\quad  g(s,t,\gamma)&:=(1-s)^2\|u_{0}\|^2+(1-t)^2\|
						v_{0}\|^2 +\beta^2|\gamma-\gamma_{0}|^2\label{e:sunday1}\\
						\textrm{subject to}\quad  g_{1}(s,t,\gamma)&:=s^2\norm{u_{0}}^2 -
						t^2\norm{v_{0}}^2-2\alpha\gamma =0,\quad s\geq 0, t\geq 0,\gamma\in\RR.\label{e:sunday2}
					\end{align}
				\end{subequations}
				Because $g$ is continuous and coercive, and $g_{1}$ is
				continuous, we conclude that the optimization problem \cref{e:sunday} 
				has a solution. 
			\end{proof}
			
			Next, we provide a result on set convergence and review graphical convergence, see, e.g., \cite{rockafellar_variational_1998,
				aubin2009set}.
			We shall need the \emph{cross} 
			\begin{empheq}[box=\mybluebox]{equation}
				C:= \menge{(x,y)\in X \times X}{\scal{x}{y}=0},
			\end{empheq}
			which was studied in, e.g., \cite{bauschke2022projection},
			as well as 
			\begin{empheq}[box=\mybluebox]{equation}
				\Ct:=\menge{(u,v)\in X\times X}{\|u\|^2-\|v\|^2=0}.
			\end{empheq}
			
			\begin{proposition}\label{p:setconvergence} The following hold:
				\begin{enumerate}
					\item\label{i:uv}
					$\lim_{\alpha\rightarrow 0}\tCalp =\Ct\times\RR$.
					\item\label{i:xy}
					$\lim_{\alpha\rightarrow 0}\Calp=C\times\RR$.
				\end{enumerate}
			\end{proposition}
			
			\begin{proof} \cref{i:uv}: 
				First we show that $\limsup_{\alpha\rightarrow 0}\tCalp\subseteq \Ct\times\RR $.
				Let $(u_\alpha, v_{\alpha}, \gamma_{\alpha})\rightarrow (u,v,\gamma)$ and $(u_\alpha, v_{\alpha}, \gamma_{\alpha})\in
				\tCalp$ with $\alpha\rightarrow 0$. Then
				$\|u_{\alpha}\|^2-\|v_{\alpha}\|^2=2\alpha\gamma_{\alpha}$ gives
				$\|u\|^2-\|v\|^2=0$ when $\alpha\rightarrow 0$, so $(u,v,\gamma)\in\Ct\times\RR$.
				
				Next we show $\Ct\times\RR\subseteq \liminf_{\alpha\rightarrow 0}\tCalp$. Let $(u,v,\gamma)\in\Ct\times\RR$,
				i.e., $\|u\|^2-\|v\|^2=0$ and $\gamma\in\RR$.
				Let $\varepsilon>0$. 
				We consider three cases:
				
				Case 1: $\gamma=0$. Then $(u_{\alpha}, v_{\alpha}, 0)=(u,v,0)\in\tCalp$ for
				every $\alpha$.
				
				Case 2: $\gamma\neq 0$ but $(u,v)=(0,0).$ If $\alpha\gamma>0$, take
				$(u_{\alpha},0,\gamma)$ with $\|u_{\alpha}\|^2-0=\alpha\gamma$ so that
				$(u_{\alpha}, 0, \gamma)\in C_{\alpha}$; if $\alpha\gamma <0$, take
				$(0,v_{\alpha},\gamma)$ with $0-\|v_{\alpha}\|^2=\alpha\gamma$ so that
				$(0, v_{\alpha},\gamma)\in C_{\alpha}$. Then
				$$\|(u_{\alpha}, 0, \gamma)-(0,0,\gamma)\|=\|u_{\alpha}\|=\sqrt{|\alpha\gamma|}<\varepsilon,$$
				or
				$$\|(0, v_{\alpha},\gamma)-(0,0,\gamma)\|=\|v_{\alpha}\|=\sqrt{|\alpha\gamma|}<\varepsilon,$$
				if $|\alpha|<\varepsilon^2/|\gamma|.$
				
				Case 3: $\gamma\neq 0$ and $(u,v)\neq (0,0)$. 
				Take $\alpha\in\RR$ such that 
				$$|\alpha|<\min\left\{\frac{\varepsilon\|(u,v)\|}{|\gamma|}, \frac{\|(u,v)\|^2}{|\gamma|}\right\},$$
				and set 
				$$\lambda:=\frac{\alpha\gamma}{\|(u,v)\|^2}.$$
				Then
				$$|\lambda|=\frac{|\alpha\gamma|}{\|(u,v)\|^2}<1.$$
				Now set 
				$$u_{\alpha}:=\sqrt{1+\lambda}u, \quad v_{\alpha}:=\sqrt{1-\lambda}v.$$
				Then
				\begin{align*}
					\|u_{\alpha}\|^2-\|v_{\alpha}\|^2 &= (1+\lambda)\|u\|^2-(1-\lambda)\|v\|^2\\
					&=\lambda(\|u\|^2+\|v\|^2)=\alpha\gamma,
				\end{align*}
				so that $(u_{\alpha},v_{\alpha},\gamma)\in \tCalp$
				and
				\begin{align*}
					\|(u_{\alpha}, v_{\alpha},\gamma)-(u,v,\gamma)\| &=\sqrt{(\sqrt{1+\lambda}-1)^2\|u\|^2+
						(\sqrt{1-\lambda}-1)^2\|v\|^2}\\
					&= \sqrt{\frac{\lambda^2}{(1+\sqrt{1+\lambda})^2}\|u\|^2+
						\frac{\lambda^2}{(1+\sqrt{1-\lambda})^2}\|v\|^2}\\
					&\leq \sqrt{\lambda^2(\|u\|^2+\|v\|^2)}=|\lambda|\|(u,v)\|<\varepsilon.
				\end{align*}
				
				\cref{i:xy}: This follows from \ref{i:uv} because that $\Calp=A\tCalp$ and $C\times\RR=A(\Ct\times \RR)
				$ and that $A$ is an isometry.  
				See also \cite[Theorem~4.26]{rockafellar_variational_1998}. 
			\end{proof}

			\begin{definition}
				{\bf (graphical limits of mappings)} (See \cite[Definition~5.32]{rockafellar_variational_1998}.)
				For a sequence of set-valued mappings $S^{k}:\RR^n\rightrightarrows \RR^m$,
				we say $S^k$ converges graphically to $S$, in symbols $S^{k}\graphc S$, if for every $x\in\RR^n$ one has
				$$\bigcup_{\{x^{k}\rightarrow x\}}\limsup_{k\rightarrow\infty}
				S^{k}(x^k)\subseteq S(x)\subseteq \bigcup_{\{x^{k}\rightarrow x\}}\liminf_{k\rightarrow\infty}
				S^{k}(x^k).$$
			\end{definition}

			\begin{fact} {\bf (Rockafellar--Wets)} 
				(See \cite[Example~5.35]{rockafellar_variational_1998}.)\label{f:projconvergence}
				For closed subsets sets $S^{k}, S$ of $\RR^n$, 
				one has $P_{S^{k}} \graphc P_{S}$ if and only if 
				$S^ {k}\to S$.
			\end{fact}

			We are now ready for our main results which we will derive in the next section. 
			
			\section{Projection onto a rectangular hyperbolic paraboloid}

			\label{sec:main}
			
			We begin with projections onto rectangular hyperbolic paraboloids.
			In view of \cref{p:transform}\cref{i:proj}, to find
			$\proj{\Calp}$ it suffices to find $\proj{\tCalp}$. That is,
			for every $(u_0, v_{0},\gamma_{0})\in X\times X\times \RR$,
			we need to solve:
			\begin{subequations}
				\label{e:findproj}
				\begin{align}
					\min_{u,v,\gamma} \quad   f(u,v,\gamma)&:=\norm{u-u_0}^2 + \norm{v-v_0}^2 + \beta^2 |\gamma-\gamma_{0}|^2\\
					\textrm{subject to} \quad  h(u,v,\gamma)&:=\norm{u}^2 -\norm{v}^2-2\alpha \gamma=0.
				\end{align}
			\end{subequations}
			
			\begin{theorem}
				\label{p:2011d4.10}
				Let $(u_0,v_0,\gamma_{0}) \in X\times X\times\RR$. Then the following hold:
				\begin{enumerate}
					\item \label{p:2011d4.10a}
					When $u_0 \neq 0,v_0\neq 0$, then 
					\begin{align}
						\label{e:proj_formula}
						\proj{\tCalp}(u_0,v_0,\gamma_{0}) 
						= 
						\bigg\{ \Big(\dfrac{u_0}{1+\lambda}, \dfrac{v_0}{1-\lambda}, \gamma_0+\dfrac{\lambda\alpha}{\beta^2}\Big)\bigg\},
					\end{align}
					where the unique
					$\lambda \in\left]-1,1\right[$ solves the following (essentially) 
					quintic equation
					\begin{align}
						\label{eq:quintic}
						g(\lambda):= \frac{(\lambda^2 + 1 ) p - 2\lambda q}{(1-\lambda^2)^2} -
						\frac{2\lambda\alpha^2}{\beta^2} - 2\alpha \gamma_0= 0,
					\end{align}
					and where $p :=\norm{u_0}^2 - \norm{v_0}^2$ and $q :=\norm{u_0}^2 + \norm{v_0}^2$.
					
					\item  \label{p:2011d4.10b}
					When $u_0 = 0,v_0\neq 0$, we have:
					\begin{enumerate}
						\item \label{p:2011d4.10b(a)} If $\alpha(\gamma_{0}-\frac{\alpha}{\beta^2})<-\frac{\|v_{0}\|^2}{8}$, then 
						\begin{equation}
							\label{e:220605a} \proj{\tCalp}(0,v_0,\gamma_0) =
							\bigg\{\Big(
							0,\dfrac{v_{0}}{1-\lambda}, \gamma_{0}+\dfrac{\lambda\alpha}{\beta^2}
							\Big) \bigg\}, 
						\end{equation}
						for a unique $\lambda\in\left]-1, 1\right[$ that
						solves the (essentially) cubic equation
						\begin{align}
							\label{eq:cubic1}
							g_1(\lam) := \frac{\|v_{0}\|^2}{(1-\lambda)^2}+\frac{2\lambda\alpha^2}{\beta^2}
							+2\alpha\gamma_{0}
							=0.
						\end{align}
						
						\item \label{p:2011d4.10b(b)} If  $\alpha(\gamma_{0}-\frac{\alpha}{\beta^2})\geq -\frac{\|v_{0}\|^2}{8}$, then 
						\begin{equation}
							\label{e:220605b}\proj{\tCalp}(0,v_0,\gamma_0)=
							\mmenge{\Big(u,\frac{v_0}{2},\gamma_0-\frac{\alpha}{\beta^2}\Big)}{ 
								\|u\|=\sqrt{{2\alpha\Big(\gamma_{0}-\frac{\alpha}{\beta^2}\Big)+\frac{\|v_{0}\|^2}{4}}}, u\in X}, 
						\end{equation}
						which is a singleton if and only if 
						$\alpha(\gamma_{0}-\frac{\alpha}{\beta^2})= -\frac{\|v_{0}\|^2}{8}$. 
					\end{enumerate}
					
					\item 
					\label{p:2011d4.10c}
					When $u_0\neq 0, v_0 = 0$, we have:
					\begin{enumerate}
						\item If $\alpha(\gamma_{0}+\frac{\alpha}{\beta^2})>\frac{\|u_{0}\|^2}{8}$, then 
						\begin{equation}  \proj{\tCalp}(u_0,0,\gamma_0) =
							\bigg\{\Big(\dfrac{u_{0}}{1+\lambda},0, \gamma_{0}+\dfrac{\lambda\alpha}{\beta^2}\Big) \bigg\}
						\end{equation}
						for a unique $\lambda\in\left]-1, 1\right[$ that solves
						the (essentially) cubic equation 
						\begin{align}
							\label{eq:cubic2}
							g_2(\lam) := \frac{\|u_{0}\|^2}{(1+\lambda)^2}-\frac{2\lambda\alpha^2}{\beta^2}-2\alpha\gamma_{0}=0.
						\end{align}
						\item 
						If  $\alpha(\gamma_{0}+\frac{\alpha}{\beta^2})\leq \frac{\|u_{0}\|^2}{8}$, 
						then 
						\begin{equation}
							\proj{\tCalp}(u_0,0,\gamma_0)=
							\mmenge{\Big(\dfrac{u_{0}}{2},v,\gamma_{0}+\dfrac{\alpha}{\beta^2} \Big)}{\|v\|=\sqrt{-2\alpha\Big(\gamma_{0}+\frac{\alpha}{\beta^2}\Big)+\frac{\|u_{0}\|^2}{4}}, v\in X},
						\end{equation}
						which is a singleton if and only if 
						$\alpha(\gamma_{0}+\frac{\alpha}{\beta^2})= \frac{\|u_{0}\|^2}{8}$.
					\end{enumerate}
					\item \label{p:2011d4.10d}
					When $u_0=0, v_0=0$, we have: 
					\begin{enumerate}
						\item \label{d4.10da} If $\alpha\gamma_0>\frac{\alpha^2}{\beta^2}$,
						then the projection is the non-singleton set 
						\begin{equation}
							\label{e:220605c}
							\proj{\tCalp}(0,0,\gamma_0)=
							\mmenge{\Big(u,0, \gamma_0-\dfrac{\alpha}{\beta^2}\Big)}{\ \|u\|=
								\sqrt{2\alpha\Big(\gamma_0-\frac{\alpha}{\beta^2}\Big)},\, u\in X}
							.
						\end{equation}
						\item \label{d4.10db} 
						If $|\alpha\gamma_0|\leq \frac{\alpha^2}{\beta^2}$,
						then 
						\begin{equation}\proj{\tCalp}(0,0,\gamma_0)=
							\big\{(0,0,0) \big\}.\end{equation}
						\item \label{d4.10de} If $\alpha\gamma_0<-\frac{\alpha^2}{\beta^2}$,
						then the projection is the non-singleton set
						\begin{equation}\proj{\tCalp}(0,0,\gamma_0)=
							\mmenge{\Big(0,v,\gamma_0+\dfrac{\alpha}{\beta^2} \Big)}{\|v\|=\sqrt{-2\alpha\Big(\gamma_0+\frac{\alpha}{\beta^2}\Big)},\, v\in X}.
						\end{equation}
					\end{enumerate}
				\end{enumerate}
			\end{theorem}
			\begin{proof}
				Observe that $\nabla f(u,v,\gamma) = (2(u-u_{0}),2(v-v_{0}),2\beta^2(\gamma-\gamma_0))$ and $\nabla h(u,v,\gamma)=(2u,-2v,-2\alpha)$.
				Since $\alpha\neq0$,
				we have $\forall (u,v,\gamma) \in X\times X \times \RR,~\nabla h(u,v,\gamma) \neq 0$.
				Using \cite[Proposition~4.1.1]{bertsekas1997nonlinear}, we obtain
				the following KKT
				optimality conditions of \cref{e:findproj}:
				\begin{subequations}
					\label{mall}
					\begin{align}
						(1+\lambda)u=u_{0}\label{m2}\\
						(1-\lambda)v=v_{0}\label{m3}\\
						\beta^2(\gamma-\gamma_{0}) - \lambda \alpha = 0\label{m4}\\
						\norm{u}^2 - \norm{v}^2-2\alpha \gamma = 0 \label{m1}
					\end{align}
				\end{subequations}
				where $\lambda\in\RR$ is the Lagrange multiplier. 
				
				The proofs of \cref{p:2011d4.10a}--\cref{p:2011d4.10d} are 
				presented in \cref{sub:5.1}--\cref{sub:5.4} below.
				
				\subsection{Case \cref{p:2011d4.10a}: $u_0 \neq 0,v_0\neq 0$}
				\label{sub:5.1}
				\begin{proof}
					Because  $u_0 \neq 0, v_0 \neq 0$, we obtain $\lambda \neq \pm 1$. Solving \cref{m2}, \cref{m3} and \cref{m4} gives
					$u = \frac{u_0}{(1+\lambda)}$, $v = \frac{v_0}{(1-\lambda)}$ and $\gamma = \gamma_0 + \frac{\lambda\alpha}{\beta^2} $. 
					By \cref{p:three:dim}, $1+\lambda>0$ and $1-\lambda>0$, i.e.,
					$\lambda\in \left]-1,1\right[$. 
					Substituting $u$ and $v$ back into equation \cref{m1}, we get the 
					(essentially) quintic equation \cref{eq:quintic}.
					Using also $p<q$ and $q>0$, we have 
					\begin{align*}
						(\forall \lambda\in \left]-1,1\right[)\ g'(\lambda)
						&=\frac{2}{(1-\lambda^2)^3}\big(-q(1+3\lambda^2) + 
						p(\lambda^3+3\lambda)\big)
						- 2\frac{\alpha^2}{\beta^2}\\
						&<\frac{2}{(1-\lambda^2)^3}\big(-q(1+3\lambda^2) + q(\lambda^3+3\lambda)\big)
						- 2\frac{\alpha^2}{\beta^2}\\
						&=\frac{2q(\lambda-1)^3}{(1-\lambda^2)^3}
						- 2\frac{\alpha^2}{\beta^2}
						=\frac{-2q}{(1+\lambda)^3}
						- 2\frac{\alpha^2}{\beta^2}\\
						&<0; 
					\end{align*}
					hence, $g$ is strictly decreasing. 
					Moreover,  $g(-1) = +\infty$, $g(1) =-\infty$ and $g$ is continuous on $\left]-1,1\right[$.
					Thus, $g(\lambda) = 0$ has unique zero in $\left]-1,1\right[$.
				\end{proof}\\
				
				\subsection{Case \cref{p:2011d4.10b}: $u_0 = 0, v_0 \neq 0$}
				\label{sub:5.2}
				\begin{proof}
					When $u_0 = 0$, the objective function is
					$$f(u,v,\gamma)=\|u\|^2+\|v-v_{0}\|^2+\beta^2|\gamma-\gamma_{0}|^2,$$
					and the KKT optimality conditions \cref{mall} become
					\begin{subequations}
						\label{mallu}
						\begin{align}
							(1+\lambda)u&=0\label{m4u}\\
							(1-\lambda)v&=v_{0}\label{m3u}\\
							\gamma&=\gamma_{0}+\frac{\lambda\alpha}{\beta^2}\label{m2u}\\
							\norm{u}^2 - \norm{v}^2&=2\alpha \gamma.\label{m1u}
						\end{align}
					\end{subequations}
					Then \cref{m4u} gives 
					\begin{equation}
						\label{e:220604a} 
						1+\lambda=0\;\;\text{or}\;\; u=0.
					\end{equation}
					Because $v_{0}\neq 0$, we have
					$1-\lambda\neq 0$, so that 
					\begin{equation}
						\label{e:220604b}
						v=\frac{v_{0}}{1-\lambda}.
					\end{equation}
					By \cref{p:three:dim}, $\lambda <1$.
					
					Our analysis is divided into the following three situations:
					
					\noindent\textbf{Situation 1: $\alpha(\gamma_{0}-\frac{\alpha}{\beta^2})<-\frac{\|v_{0}\|^2}{8}$.}
					
					In view of \cref{e:220604a}, we analyze two cases.
					
					\noindent \textbf{Case~1:}~$1+\lambda=0$, i.e., $\lambda =-1$. 
					By \cref{e:220604b},
					$v = \frac{v_0}{2}$, and then \cref{m1u} and \cref{m2u} give
					\begin{align*}
						\|u\|^2& =2\alpha\gamma+\frac{\|v_0\|^2}{4}=2\alpha\Big(\gamma_{0}-\frac{\alpha}{\beta^2}\Big)+\frac{\|v_0\|^2}{4}<0,
					\end{align*}
					which is absurd.
					
					\noindent \textbf{Case 2:} $u=0$. 
					By \cref{m1u}, $-\|v\|^2=2\alpha\gamma$, 
					together with \cref{e:220604b} and \cref{m2u}, we have
					$$g_1(\lambda):=\frac{\|v_{0}\|^2}{(1-\lambda)^2}+2\alpha\Big(\gamma_{0}+\frac{\lambda\alpha}{\beta^2}\Big)=0.$$
					As
					$$g_1'(\lambda)=\frac{2\|v_{0}\|^2}{(1-\lambda)^3}+\frac{2\alpha^2}{\beta^2}>0 \;\;\text{ on }\; \left]-\infty, 1\right[,$$
					$g_1$ is strictly increasing on $\left]-\infty, 1\right[$.
					Moreover,
					$g_1(1)=+\infty$ and
					$$g_1(-1)=\frac{\|v_{0}\|^2}{4}+2\alpha\Big(\gamma_{0}-\frac{\alpha}{\beta^2}\Big)
					<0.$$
					Because $g_1$ is strictly increasing and continuous, by the Intermediate Value Theorem,
					there exists a unique $\lambda\in\ ]-1, 1[$ such that $g_1(\lambda)=0$.
					Hence, the possible optimal solution is given by
					\begin{equation}\label{e:case2win}
						\Big(0, \frac{v_{0}}{1-\lambda}, \gamma_{0}+\frac{\lambda\alpha}{\beta^2}\Big),
					\end{equation}
					where $g_1(\lambda)=0$ and $\lambda\in\left]-1,1\right[$.
					
					Combining Case 1 and Case 2, we obtain that \cref{e:case2win} is the unique projection.

					\noindent\textbf{Situation 2: $\alpha(\gamma_{0}-\frac{\alpha}{\beta^2})>-\frac{\|v_{0}\|^2}{8}$.}
					
					In view of \cref{e:220604a}, we consider two cases: 
					
					\noindent \textbf{Case 1:}~$1+\lambda=0$, i.e., $\lambda =-1$.
					By \cref{e:220604b},
					$v = \frac{v_0}{2}$, and then \cref{m1u} and \cref{m2u} give
					\begin{align*}
						\|u\|^2& =2\alpha\gamma+\frac{\|v_0\|^2}{4}=2\alpha\Big(\gamma_{0}-\frac{\alpha}{\beta^2}\Big)+\frac{\|v_0\|^2}{4}>0.
					\end{align*}
					The possible optimal value is attained at
					\begin{equation}\label{e:winner}
						\Big(u, \frac{v_{0}}{2}, \gamma_{0}-\frac{\alpha}{\beta^2}\Big)
					\end{equation}
					with $\|u\|^2=2\alpha(\gamma_{0}-\frac{\alpha}{\beta^2})+\frac{\|v_{0}\|^2}{4}$ such that
					\begin{equation}\label{e:obj1}
						f\Big(u,\frac{v_{0}}{2},\gamma_{0}-\frac{\alpha}{\beta^2}\Big)=2\alpha\gamma_{0}
						-\frac{\alpha^2}{\beta^2}+\frac{\|v_{0}\|^2}{2}.
					\end{equation}

					\noindent \textbf{Case 2:} $u=0$. By \cref{m1u}, $-\|v\|^2=2\alpha\gamma$, together with \cref{m2u}, 
					we have
					$$g_1(\lambda):=\frac{\|v_{0}\|^2}{(1-\lambda)^2}+2\alpha\Big(\gamma_{0}+\frac{\lambda\alpha}{\beta^2}\Big)=0.$$
					As
					$$g_1'(\lambda)=\frac{2\|v_{0}\|^2}{(1-\lambda)^3}+\frac{2\alpha^2}{\beta^2}>0 \;\;\text{ on } \left]-\infty, 1\right[,$$
					$g_1$ is strictly increasing. 
					Observe that
					$$g_1(-1)=\frac{\|v_0\|^2}{4}+2\alpha\Big(\gamma_{0}-\frac{\alpha}{\beta^2}\Big)
					>0,$$
					and
					$g_1(-\infty)=-\infty$. By the Intermediate Value Theorem, 
					there exists a unique $\lambda\in\left]-\infty, -1\right[$
					such that $g_1(\lambda)=0$
					because $g_1$ is strictly increasing and continuous.
					The possible optimal value is attained at (recall \cref{e:220604b})
					$$\Big(0, \frac{v_{0}}{1-\lambda}, \gamma_{0}+\frac{\lambda\alpha}{\beta^2}\Big)$$ with
					\begin{equation}\label{e:obj22}
						f\Big(0,\frac{v_{0}}{1-\lambda},\gamma_{0}+\frac{\lambda\alpha}{\beta^2}\Big)=
						\frac{\lambda^2\|v_{0}\|^2}{(1-\lambda)^2}+\frac{\lambda^2\alpha^2}{\beta^2},
					\end{equation}
					where $\lambda$ is the unique solution of
					\begin{equation}\label{e:lamb}
						g_1(\lambda):=\frac{\|v_{0}\|^2}{(1-\lambda)^2}+2\alpha\Big(\gamma_{0}+\frac{\lambda\alpha}{\beta^2}\Big)=0\;\;
						\text{ in $\left]-\infty, -1\right[$.}
					\end{equation}
					
					Because both Case~1 and Case~2 may occur, we have to compare possible optimal objective function values,
					namely, \cref{e:obj1} and \cref{e:obj22}.
					We claim that Case~1 wins, i.e.,
					\begin{equation}\label{e:compare}
						2\alpha\gamma_{0}
						-\frac{\alpha^2}{\beta^2}+\frac{\|v_{0}\|^2}{2}<
						\frac{\lambda^2\|v_{0}\|^2}{(1-\lambda)^2}+\frac{\lambda^2\alpha^2}{\beta^2}.
					\end{equation}
					In view of \cref{e:lamb}, we have
					\begin{equation}\label{e:gv0}
						0<\frac{\|v_{0}\|^2}{(1-\lambda)^2}=-2\alpha\Big(\gamma_{0}+\frac{\lambda\alpha}{\beta^2}\Big),
						\;\text{ and so }\; \alpha\Big(\gamma_{0}+\frac{\lambda\alpha}{\beta^2}\Big)< 0.
					\end{equation}
					To show \cref{e:compare}, we shall reformulate it in equivalent forms:
					\begin{equation*}
						\bigg(\lambda^2-\frac{(1-\lambda)^2}{2}\bigg)\frac{\|v_{0}\|^2}{(1-\lambda)^2}
						+(1+\lambda^2)\frac{\alpha^2}{\beta^2}> 2\alpha\gamma_{0},
					\end{equation*}
					which is
					$$
					\frac{\lambda^2+2\lambda-1}{2}\bigg(-2\alpha\Big(\gamma_{0}+\frac{\lambda\alpha}{\beta^2}\Big)\bigg)
					+(1+\lambda^2)\frac{\alpha^2}{\beta^2}> 2\alpha\gamma_{0}$$
					by \cref{e:gv0}.
					After simplifications,
					this reduces to
					$$\frac{\alpha^2}{\beta^2}(1+\lambda)^2(1-\lambda)>\alpha\gamma_{0}(1+\lambda)^2.$$
					Since $\lambda+1<0$, this is equivalent to
					$$\frac{\alpha^2}{\beta^2}(1-\lambda)>\alpha\gamma_{0}, \;\;\text{ i.e., } \;\alpha\Big(\gamma_{0}+\frac{\lambda\alpha}{\beta^2}\Big)<\frac{\alpha^2}{\beta^2},$$
					which obviously holds because of \cref{e:gv0} and $\alpha^2/\beta^2>0$.
					
					Hence, equation \cref{e:winner} of Case~1 gives the optimal solution.
					
					\noindent\textbf{Situation 3: }
					\begin{equation}\label{e:negativeone}
						\alpha\Big(\gamma_{0}-\frac{\alpha}{\beta^2}\Big)=-\frac{\|v_{0}\|^2}{8}.
					\end{equation}
					We again consider two cases.
					
					\noindent \textbf{Case 1:}~$1+\lambda=0$, i.e., $\lambda =-1$.
					By \cref{m3u},
					$v = \frac{v_0}{2}$ and then \cref{m1u} and \cref{m2u} give
					\begin{align*}
						\|u\|^2& =2\alpha\gamma+\frac{\|v_0\|^2}{4}
						=2\alpha\Big(\gamma_{0}-\frac{\alpha}{\beta^2}\Big)+\frac{\|v_0\|^2}{4}=0,
					\end{align*}
					so $u=0$.
					The possible optimal value is attained at
					\begin{equation}\label{e:coincide1}
						\Big(0, \frac{v_{0}}{2}, \gamma_{0}-\frac{\alpha}{\beta^2}\Big)
					\end{equation}
					with
					\begin{equation*}
						f\Big(0,\frac{v_{0}}{2},\gamma_{0}-\frac{\alpha}{\beta^2}\Big)
						=
						\frac{\|v_{0}\|^2}{4}
						+
						\frac{\alpha^2}{\beta^2}.
					\end{equation*}
					
					\noindent \textbf{Case 2:} $u=0$. 
					By \cref{m1u}, $-\|v\|^2=2\alpha\gamma$, together with \cref{m2u}, 
					we have
					$$g_2(\lambda):=\frac{\|v_{0}\|^2}{(1-\lambda)^2}
					+2\alpha\Big(\gamma_{0}+\frac{\lambda\alpha}{\beta^2}\Big)=0.$$
					By \cref{e:negativeone},
					$$g_2(-1)=\frac{\|v_{0}\|^2}{4}+
					2\alpha\Big(\gamma_{0}-\frac{\alpha}{\beta^2}\Big)=0.$$
					As
					$$g_2'(\lambda)=\frac{2\|v_{0}\|^2}{(1-\lambda)^3}+\frac{2\alpha^2}{\beta^2}>0 \;\;\text{ on }\; \left]-\infty, 1\right[,$$
					$g_2$ is strictly increasing and continuous on $\left]-\infty,1\right[$, 
					so $\lambda=-1$ is the unique solution in $\left]-\infty, 1\right[$.
					Then the possible optimal value is attained at
					$$\Big(0, \frac{v_{0}}{2}, \gamma_{0}-\frac{\alpha}{\beta^2}\Big)$$ with
					\begin{equation}\label{e:obj2}
						f\Big(0,\frac{v_{0}}{2},\gamma_{0}-\frac{\alpha}{\beta^2}\Big)=
						\frac{\|v_{0}\|^2}{4}+\frac{\alpha^2}{\beta^2}.
					\end{equation}
					
					Therefore, Case~1 and Case~2 give exactly the same solution. 
					The optimal solution is given by
					\cref{e:coincide1}, and it can be recovered by \cref{e:winner}, 
					the optimal solution of Situation~2.
				\end{proof}

				\subsection{Case \cref{p:2011d4.10c}: $u_0\neq0, v_0 = 0$}
				\label{sub:5.3}
				\begin{proof}
					The minimization problem now is
					\begin{subequations}
						\label{e:original}
						\begin{align}
							\text{minimize}~~  \ f(u,v,\gamma)&=\|u_{0}-u\|^2+\|v\|^2+\beta^2
							|\gamma_{0}-\gamma|^2\\
							\text{ subject to}~~ \ \|u\|^2-\|v\|^2&=2\alpha\gamma.
						\end{align}
					\end{subequations}
					Rewrite it as
					\begin{subequations}
						\label{e:switch}
						\begin{align}
							\text{minimize}~~  \ f(u,v,\gamma)&=\|v\|^2+\|u_{0}-u\|^2+\beta^2 |\gamma_{0}-\gamma|^2\\
							\text{ subject to}~~ \  \|v\|^2-\|u\|^2&=2(-\alpha)\gamma.
						\end{align}
					\end{subequations}
					Luckily, we can apply \cref{sub:5.2} for the point $(0, u_{0}, \gamma_{0})$ and parameter
					$-\alpha$. 
					More precisely, when $-\alpha(\gamma_{0}-\frac{-\alpha}{\beta^2})<-\frac{\|u_{0}\|^2}{8}$, the optimal solution to
					\cref{e:switch} is
					$$\Big(0, \frac{u_{0}}{1-\tilde{\lambda}}, \gamma_{0}+\frac{\tilde{\lambda}(-\alpha)}{\beta^2}\Big)$$
					where
					$\tilde{g}_2(\tilde{\lambda})=0$, $\tilde{\lambda}\in\left]-1,1\right[$, 
					and
					$$\tilde{g}_2(\tilde{\lambda})=\frac{\|u_{0}\|^2}{(1-\tilde{\lambda})^2}
					+2(-\alpha)\Big(\gamma_{0}-\frac{\tilde{\lambda}\alpha}{\beta^2}\Big)=0.$$
					Put $\lambda=-\tilde{\lambda}$. 
					Simplifications give:
					when $\alpha(\gamma_{0}+\frac{\alpha}{\beta^2})>\frac{\|u_{0}\|^2}{8}$, 
					the optimal solution to \cref{e:switch} 
					is
					\begin{equation}\label{e:s:proj1}
						\Big(0,\frac{u_{0}}{1+{\lambda}}, \gamma_{0}+\frac{{\lambda}\alpha}{\beta^2}\Big)
					\end{equation}
					where
					$g_2({\lambda})=0$, ${\lambda}\in \left]-1,1\right[$, and
					$$g_2(\lambda) := \tilde{g}_2(-{\lambda})=\frac{\|u_{0}\|^2}{(1+{\lambda})^2}
					-2\alpha\Big(
					\gamma_{0}+\frac{{\lambda}\alpha}{\beta^2}\Big)=0.$$
					Switching the first and second components in \cref{e:s:proj1}
					gives the optimal solution to \cref{e:original}.

					When $-\alpha(\gamma_{0}-\frac{-\alpha}{\beta^2})\geq -\frac{\|u_{0}\|^2}{8}$,
					the optimal solution to
					\cref{e:switch} is
					$$\Big(v, \frac{u_{0}}{2}, \gamma_{0}-\frac{-\alpha}{\beta^2}\Big)$$
					with
					$$\|v\|^2=2(-\alpha)\Big(\gamma_{0}-\frac{-\alpha}{\beta^2}\Big)
					+\frac{\|u_{0}\|^2}{4}.$$
					That is,
					when $\alpha(\gamma_{0}+\frac{\alpha}{\beta^2})\leq \frac{\|u_{0}\|^2}{8}$,
					the optimal solution to \cref{e:switch} is
					\begin{equation}\label{e:s:proj2}
						\Big(v, \frac{u_{0}}{2}, \gamma_{0}+\frac{\alpha}{\beta^2}\Big)
					\end{equation}
					with
					$$\|v\|^2=-2\alpha\Big(\gamma_{0}+\frac{\alpha}{\beta^2}\Big)+\frac{\|u_{0}\|^2}{4}.$$
					Switching the first and second components in \cref{e:s:proj2}
					gives the optimal solution to \cref{e:original}.
				\end{proof}
				
				\subsection{Case \cref{p:2011d4.10d}: $u_0 = v_0 = 0$}
				\label{sub:5.4}
				\begin{proof}
					The objective function is $f(u,v,\gamma)=\|u\|^2+\|v\|^2+\beta^2|\gamma-\gamma_0|^2$, and
					the KKT optimality conditions \cref{mall} become
					\begin{subequations}
						\label{Cletters}
						\begin{align}
							\quad(1+\lam)u &=0, \label{Ca}\\
							\quad(1-\lam)v &=0,\label{Cb}\\
							\gamma&=\gamma_0+\frac{\lam\alpha}{\beta^2},\label{Cc}\\
							\|u\|^2-\|v\|^2 & =2\alpha\gamma.\label{Cd}
						\end{align}
					\end{subequations}
					We shall consider three cases:
					\begin{enumerate}
						\item \label{C1} $\alpha(\gamma_0-\frac{\alpha}{\beta^2})>0$; hence, 
						$\gamma_0-\frac{\alpha}{\beta^2}\neq0$.
						\item \label{C2} $\alpha(\gamma_0-\frac{\alpha}{\beta^2})=0$; hence, 
						$\gamma_0-\frac{\alpha}{\beta^2}=0$.
						\item \label{C3}
						$\alpha(\gamma_0-\frac{\alpha}{\beta^2})<0$; hence, 
						$\gamma_0-\frac{\alpha}{\beta^2}\neq0$.
					\end{enumerate}
					For each item \cref{C1}--\cref{C3}, we will apply \cref{Cletters}:\\
					\textbf{Case 1:} $\alpha(\gamma_0-\frac{\alpha}{\beta^2})>0$. 
					By \cref{Ca}, we have $\lam=-1$ or $u=0$.
					We consider two subcases.
					
					\textbf{Subcase~1:}~$\lam=-1$. Using \cref{Cb}, \cref{Cc} and \cref{Cd}, we obtain $v=0$, $\gamma=\gamma_0-\frac{\alpha}{\beta^2}$, and
					\begin{align}
						\label{uarc}
						\|u\|^2=2\alpha\Big(\gamma_0-\frac{\alpha}{\beta^2}\Big).
					\end{align}
					Therefore, the candidate for the solution is $(u,0,\gamma_0-\frac{\alpha}{\beta^2})$ with $u$ given by \cref{uarc} and its 
					objective function value is
					\begin{align}
						\label{obj1}
						f\Big(u,0,\gamma_0-\frac{\alpha}{\beta^2}\Big)
						= 2\alpha\Big(\gamma_0-\frac{\alpha}{\beta^2}\Big)+0
						+\beta^2\Big(\frac{-\alpha}{\beta^2}\Big)^2
						=2\alpha\gamma_0-\frac{\alpha^2}{\beta^2}.
					\end{align}
					
					\textbf{Subcase~2:}~$u=0$. Using \cref{Cb}--\cref{Cd}, we obtain
					$-\|v\|^2=2\alpha\gamma, \gamma=\gamma_0+\frac{\lam\alpha}{\beta^2}$ and $(1-\lam)v=0$. 
					We have to consider two further cases: $1-\lam=0$ or $v=0$.
					\begin{enumerate}
						\item $v=0$.  We get $-(0)^2=2\alpha\gamma \Rightarrow \gamma = 0$
						because $\alpha \neq 0$. This gives a possible
						solution $(0,0,0)$ with function value
						\begin{align}
							\label{obj2}
							f(0,0,0)=\|u\|^2+\|v\|^2+\beta^2|\gamma-\gamma_0|^2=\beta^2\gamma_0^2.
						\end{align}
						\item $\lam=1$. We have $\gamma=\gamma_0+\frac{\alpha}{\beta^2}$ and $-\|v\|^2=2\alpha(\gamma_0+\frac{\alpha}{\beta^2})$. So, $0\leq \|v\|^2=-2\alpha(\gamma_0+\frac{\alpha}{\beta^2})$. However,
						\begin{align}
							\label{refm}
							-2\alpha\Big(\gamma_0+\frac{\alpha}{\beta^2}\Big)
							=\underbrace{-2\alpha\Big(\gamma_0-\frac{\alpha}{\beta^2}\Big)}_{<0}-\frac{4\alpha^2}{\beta^2}<0
						\end{align}
						because $\alpha(\gamma_0-\frac{\alpha}{\beta^2})>0$. 
						This contradiction shows $\lam =1$ does not happen.
					\end{enumerate}
					We now compare objective function values \cref{obj1} and \cref{obj2}:
					$$2\alpha\gamma_0-\frac{\alpha^2}{\beta^2}<\beta^2\gamma_0^2\Leftrightarrow
					\beta^2\gamma_0^2+\frac{\alpha^2}{\beta^2}-2\alpha\gamma_0>0\Leftrightarrow\Big(
					\beta\gamma_0-\frac{\alpha}{\beta}\Big)^2>0\Leftrightarrow\beta^2\Big(\gamma_0-\frac{\alpha}{\beta^2}\Big)^2>0,$$
					which holds because $\gamma_0-\frac{\alpha}{\beta^2}\neq0$. Hence, the optimal solution is $(u,0,\gamma_0-\frac{\alpha}{\beta^2})$ with $\|u\|=\sqrt{2\alpha(\gamma_0-\frac{\alpha}{\beta^2})}$. That is,
					$$\proj{\Ct_2}(0,0,\gamma_0)=
					\mmenge{\Big(u,0,\gamma_0-\frac{\alpha}{\beta^2}\Big)}{\|u\|=\sqrt{2\alpha\Big(\gamma_0-\frac{\alpha}{\beta^2}\Big)}}.
					$$
					
					\noindent \textbf{Case~2:} $\alpha(\gamma_0-\frac{\alpha}{\beta^2})=0$; 
					hence, $\gamma_0-\frac{\alpha}{\beta^2}=0$. 
					By \cref{Ca}, we have two subcases to consider.\\
					\textbf{Subcase 1:}~$\lam=-1$. We have $v=0$, $\gamma=\gamma_0-\frac{\alpha}{\beta^2}=0,~\|u\|^2=2\alpha(\gamma_0-\frac{\alpha}{\beta^2})=0$.
					The possible solution is $(0,0,0)$.\\
					\textbf{Subcase 2:}~$u=0$. We have $-\|v\|^2=2\alpha\gamma$ and $\gamma=\gamma_0+\frac{\lam\alpha}{\beta^2}$. 
					By \cref{Cb}, $v=0$ or $\lam=1$. 
					This requires us to consider two further cases. 
					For $v=0$, we get $\gamma=0$, which
					gives a possible solution $(0,0,0)$. 
					For $\lam=1$, we get 
					$\gamma=\gamma_0+\frac{\alpha}{\beta^2},~\|v\|^2=
					-2\alpha(\gamma_0+\frac{\alpha}{\beta^2})=
					\frac{-4\alpha^2}{\beta^2}<0$, which is impossible, i.e., $\lam=1$ does not happen.
					
					Both \textbf{Subcase 1} and \textbf{Subcase 2} give the same solution $(0,0,0)$. Therefore, we have the optimal solution is $(0,0,0)$, when $\alpha(\gamma_0-\frac{\alpha}{\beta^2})=0$; equivalently, when 
					$\gamma_0=\frac{\alpha}{\beta^2}$.
					
					\noindent\textbf{Case 3:}~$\alpha(\gamma_0-\frac{\alpha}{\beta^2})<0$.
					In view of \cref{Ca}, we have $\lam=-1$ or $u=0$. 
					We show that $\lam=-1$ can't happen.
					Indeed, when $\lam=-1$, by \cref{Cb}--\cref{Cc}, we have $v=0,~\gamma=\gamma_0-\frac{\alpha}{\beta^2}$, and $0\leq\|u\|^2=2\alpha\gamma=2\alpha(\gamma_0-\frac{\alpha}{\beta^2})<0$, which is impossible.
					Therefore, we consider only the case $u=0$. 
					Then \cref{Cb}--\cref{Cd} yield 
					$\|v\|^2=-2\alpha\gamma,~\gamma=\gamma_0+\frac{\lam\alpha}{\beta^2},$ and $(1-\lam)v=0$, which requires us to consider two further cases.\\
					\textbf{Subcase 1:}~$v=0$. Then $\gamma=0$. The possible
					optimal solution is $(0,0,0)$ and its objective function value is
					\begin{align}
						\label{obj3}
						f(0,0,0)=\|u\|^2+\|v\|^2+\beta^2|\gamma-\gamma_0|^2=\beta^2\gamma_0^2.
					\end{align}
					\textbf{Subcase 2:}~$\lam=1$. Then $u=0$, $\gamma=\gamma_0+\frac{\alpha}{\beta^2}$, and  $-\|v\|^2=2\alpha(\gamma_0+\frac{\alpha}{\beta^2})$. We consider three additional cases based on the sign of  $\alpha(\gamma_0+\frac{\alpha}{\beta^2})$.
					\begin{enumerate}
						\item\label{i:nothappen} $\alpha(\gamma_0+\frac{\alpha}{\beta^2})>0$. This case never happens because
						the relation $0\geq -\|v\|^2=2\alpha(\gamma_0+\frac{\alpha}{\beta^2})>0$ 
						is absurd.
						\item\label{i:onlyzero} $\alpha(\gamma_0+\frac{\alpha}{\beta^2})=0$.
						As $\alpha\neq0$, we have $\gamma_0+\frac{\alpha}{\beta^2}=0$. This gives
						$\gamma=0,u=0$ and $v=0$. 
						So the possible optimal solution is $(0,0,0)$.
						\item $\alpha(\gamma_0+\frac{\alpha}{\beta^2})<0$. We
						have $\gamma_0+\frac{\alpha}{\beta^2}\neq0$.
						The possible optimal solution is $(0,v,\gamma_0+\frac{\alpha}{\beta^2})$ with $\|v\|=\sqrt{-2\alpha(\gamma_0+\frac{\alpha}{\beta^2})}$ and
						function value
						\begin{align}
							\label{obj4}
							f\Big(0,v,\gamma_0+\frac{\alpha}{\beta^2}\Big)=
							-2\alpha\Big(\gamma_0+\frac{\alpha}{\beta^2}\Big)
							+\beta^2\Big(\frac{\alpha}{\beta^2}\Big)^2=-2\alpha\gamma_0-\frac{\alpha^2}{\beta^2}.
						\end{align}
					\end{enumerate}
					Both \cref{i:nothappen} and \cref{i:onlyzero} 
					imply that $(0,0,0)$ from \textbf{Subcase~1}
					is the only optimal solution,
					when $\alpha^2/\beta^2>\alpha\gamma_{0}\geq -\alpha^2/\beta^2$.
					
					When $\alpha\gamma_0<-\frac{\alpha^2}{\beta^2}$, both \textbf{Subcase 1}
					and \textbf{Subcase 2} happen. We have to
					compare objectives \cref{obj3} and \cref{obj4}.
					We claim $f(0,v,\gamma_0+\frac{\alpha}{\beta^2})<f(0,0,0)$. 
					Indeed, this is equivalent to 
					$$-2\alpha\gamma_0-\frac{\alpha^2}{\beta^2}<\beta^2\gamma_{0}^2
					\Leftrightarrow \beta^2\gamma_0^2+2\alpha\gamma_0+\frac{\alpha^2}{\beta^2}>0
					\Leftrightarrow\Big(\beta\gamma_0+\frac{\alpha}{\beta}\Big)^2>0\Leftrightarrow\beta^2\Big(\gamma_0+
					\frac{\alpha}{\beta^2}\Big)^2>0$$
					which holds because $\gamma_0+\frac{\alpha}{\beta^2}\neq0$. 
					Therefore, the optimal solution is $(0,v,\gamma_0+\frac{\alpha}{\beta^2})$
					with $\|v\|=\sqrt{-2\alpha(\gamma_0+\frac{\alpha}{\beta^2})}$, i.e.,
					$$\proj{\Ct_2}(0,0,\gamma_0)=
					\mmenge{\Big(0,v,\gamma_0+\frac{\alpha}{\beta^2} \Big)}{
						\|v\|=\sqrt{-2\alpha\Big(\gamma_0+\frac{\alpha}{\beta^2}\Big)}}$$
					when
					$\alpha\gamma_0<\frac{-\alpha^2}{\beta^2}$.
				\end{proof}
				
				Altogether, \cref{sub:5.1}--\cref{sub:5.4} conclude the proof of \cref{p:2011d4.10}.
			\end{proof}

			Let us illustrate \cref{p:2011d4.10}.
			
			\begin{example}
				\label{ex:1}
				Suppose that $X=\RR$, $\alpha=5$, and $\beta=1$. 
				Writing $z$ instead of $\gamma$, we note that 
				$\tCalp$ turns into the set
				\begin{equation*}
					S := \menge{(x,y,z)\in\RR^3}{x^2-y^2=10z} 
					= \gr \big((x,y)\mapsto \tfrac{1}{10}(x^2-y^2)\big).
				\end{equation*}
				Let us now compute $P_S(x_0,y_0,z_0)$ for various points.
				\begin{enumerate}
					\item Suppose that $(x_0,y_0,z_0)=(2,-3,4)$.\\
					In view of \cref{p:2011d4.10}\cref{p:2011d4.10a}, we set
					$p := |x_0|^2-|y_0|^2 = 2^2-(-3)^2 = -5$ and 
					$q := |x_0|^2+|y_0|^2 = 2^2+(-3)^2 = 13$. 
					Following \cref{eq:quintic}, we consider the equation
					$$\frac{(\lambda^2 + 1 ) p - 2\lambda q}{(1-\lambda^2)^2} -
					\frac{2\lambda\alpha^2}{\beta^2} - 2\alpha \gamma_0=
					-\frac{5\lambda^2+26\lambda+5}{(1-\lambda^2)^2}-50\lambda-40=0
					$$
					which has 
					$\lambda =  -0.52416$ as its unique (approximate) root in $\left]-1,1\right[$.
					Using \cref{e:proj_formula} now yields
					\begin{align*}
						P_S(x_0,y_0,z_0) &= \bigg\{\Big(\dfrac{x_0}{1+\lambda}, \dfrac{y_0}{1-\lambda}, z_0+\dfrac{\lambda\alpha}{\beta^2} \Big)\bigg\}\\
						&= \Big\{\big(4.20311,-1.96830,1.37919 \big) \Big\}.
					\end{align*}
					This is depicted in \cref{fig1} with the green arrow.
					\item Suppose that $(x_0,y_0,z_0)=(0,-3,3)$.\\
					In view of 
					\cref{p:2011d4.10}\cref{p:2011d4.10b}, 
					we evaluate $\alpha(z_0-\frac{\alpha}{\beta^2})=5(3-5)=-10
					< -\frac{9}{8}=-\frac{|y_0|^2}{8}$ 
					and we are thus in case \cref{p:2011d4.10b(a)}. In view of 
					\cref{eq:cubic1}, we consider the equation
					\begin{equation*}
						\frac{|y_{0}|^2}{(1-\lambda)^2}+\frac{2\lambda\alpha^2}{\beta^2}
						+2\alpha z_{0}
						= \frac{9}{(1-\lambda)^2} + 50\lambda +30 =0
					\end{equation*}
					which has $\lambda = -0.66493$ as its unique (approximate) root
					in $\left]-1,1\right[$. 
					Using \cref{e:220605a} now yields
					\begin{align*}
						P_S(x_0,y_0,z_0)
						&= \bigg\{\Big(
						0,\dfrac{v_{0}}{1-\lambda}, \gamma_{0}+\dfrac{\lambda\alpha}{\beta^2}
						\Big) \bigg\}\\
						&= \Big\{\big(0,-1.80187,-0.32467 \big) \Big\}.
					\end{align*}
					This is depicted in \cref{fig1} with a single blue arrow. 
					\item Suppose that $(x_0,y_0,z_0)=(0,\sqrt{32},6)=(0,5.65685,6)$.\\
					In view of 
					\cref{p:2011d4.10}\cref{p:2011d4.10b}, 
					we evaluate $\alpha(z_0-\frac{\alpha}{\beta^2})=5(6-5)= 5
					> -4=-\frac{32}{8}=-\frac{|y_0|^2}{8}$ 
					and we are thus in case \cref{p:2011d4.10b(b)}.
					We compute 
					\begin{equation*}
						\sqrt{2\alpha\Big(z_0-\frac{\alpha}{\beta^2} \Big)+\frac{|y_0|^2}{4}}
						= 
						\sqrt{10(6-5)+\frac{32}{4}} = \sqrt{18}
					\end{equation*}
					and now \cref{e:220605b} yields 
					\begin{align*}
						P_S(x_0,y_0,z_0) 
						&= \mmenge{\Big(x,\frac{y_0}{2},z_0-\frac{\alpha}{\beta^2}\Big)}{ 
							|x|=\sqrt{{2\alpha\Big(z_{0}-\frac{\alpha}{\beta^2}\Big)+\frac{|y_{0}|^2}{4}}}, u\in \RR}\\
						&= \Big\{\big(\pm\sqrt{18},\sqrt{8},1 \big) \Big\}
						= \Big\{\big(\pm 4.24264 ,2.82843,1 \big) \Big\}.
					\end{align*}
					This is depicted in \cref{fig1} with double blue arrows. 
					\item Suppose that $(x_0,y_0,z_0)=(0,0,6)$.\\
					In view of 
					\cref{p:2011d4.10}\cref{p:2011d4.10d}, 
					we have $\alpha z_0 = 5(6)=30 > 25 = \frac{\alpha^2}{\beta^2}$ 
					and we are thus in case \cref{d4.10da}. 
					We compute
					\begin{equation*}
						\sqrt{2\alpha\Big(\gamma_0-\frac{\alpha}{\beta^2}\Big)}
						= \sqrt{10(6-5)} = \sqrt{10}
					\end{equation*}
					and now \cref{e:220605c} yields
					\begin{align*}
						P_S(x_0,y_0,z_0) 
						&= \mmenge{\Big(x,0,z_0-\frac{\alpha}{\beta^2}\Big)}{ 
							|x|=\sqrt{{2\alpha\Big(z_{0}-\frac{\alpha}{\beta^2}\Big)}}, u\in \RR}\\
						&= \Big\{\big(\pm\sqrt{10},0,1 \big) \Big\}
						= \Big\{\big(\pm 3.16228,0,1 \big) \Big\}.
					\end{align*}
					This is depicted in \cref{fig1} with double black arrows.
					\item Suppose that $(x_0,y_0,z_0)=(0,0,4)$.\\
					In view of 
					\cref{p:2011d4.10}\cref{p:2011d4.10d}, 
					we have $|\alpha z_0| = |5(4)|=20 < 25 = \frac{\alpha^2}{\beta^2}$ 
					and we are thus in case \cref{d4.10db}. Therefore,
					\begin{equation*}
						P_S(x_0,y_0,z_0) = \big\{(0,0,0) \big\}.
					\end{equation*}
					This is depicted in \cref{fig1} with a single black arrow.
				\end{enumerate}
			\end{example}

			\begin{figure}[ht]
				\centering
				\includegraphics[width=0.5\textwidth]{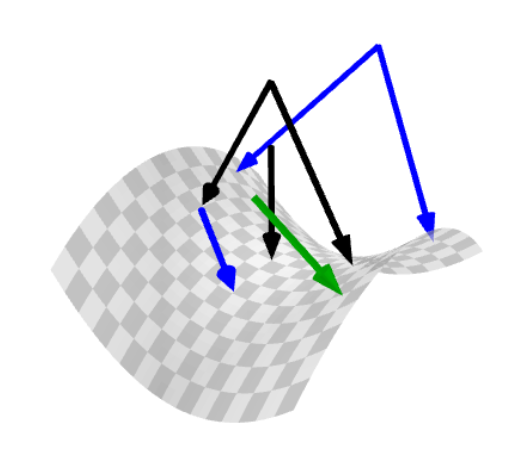}
				\caption{Visualization of the 5 projections from \cref{ex:1}.}
				\label{fig1}
			\end{figure}

			\section{Further results}
			
			\label{sec:further}
			
			Recall that
			$$\Calp= \menge{(x,y,\gamma)\in X \times X \times \RR}{\scal{x}{y} =\alpha\gamma},$$
			and this is the representation more natural to use in Deep Learning 
			(see \cite{elser2021learning}). 
			Armed with \cref{p:2011d4.10}, 
			the projection onto $\Calp$ now readily obtained:
			
			\begin{theorem}
				\label{t:2011411}
				Let $(x_{0},y_{0},\gamma_{0})\in X\times X\times \RR$. 
				Then the following hold:
				\begin{enumerate}
					\item \label{t:2011411a}
					If $x_0 \neq \pm y_0$, then 
					\begin{equation*}
						\proj{\Calp}(x_0,y_0,\gamma_0)=
						\bigg\{\Big(\dfrac{x_0 -\lambda y_0}{1-\lambda^2},\dfrac{y_0 -\lambda x_0}{1-\lambda^2},\gamma_0 +  \dfrac{\lambda\alpha}{\beta^2} \Big)\bigg\}
					\end{equation*}
					for a unique $\lambda \in\left]-1,1\right[$ that solves the 
					(essentially) quintic equation
					$$
					g(\lambda) := \frac{(\lambda^2 + 1 ) p - 2\lambda q}{(1-\lambda^2)^2}-
					\frac{2 \lambda \alpha^2}{\beta^2}-2\alpha \gamma_0= 0,$$
					where $p :=2\scal{x_0}{y_0}$ and $q :=\norm{x_0}^2 + \norm{y_0}^2.$
					
					\item \label{t:2011411b} 
					If $y_0=-x_0\neq 0$, then we have the following: \\
					a) When $\alpha(\gamma_0-\frac{\alpha}{\beta^2})< -\frac{\|x_0\|^2}{4}$, 
					then
					$$\proj{\Calp}(x_0,-x_0,\gamma_0) 
					=\bigg\{\Big(\dfrac{x_{0}}{1-\lambda},\dfrac{-x_{0}}{1-\lambda},\gamma_0+\dfrac{\lambda\alpha}{\beta^2} \Big) \bigg\}
					$$
					for a unique $\lambda\in \left]-1, 1\right[$ that solves
					$$g_1(\lambda) := \frac{2\|x_0\|^2}{(1-\lam)^2}+ \frac{2\lam\alpha^2}{\beta^2}+2\alpha \gamma_0=0.$$
					b) When $\alpha(\gamma_0-\frac{\alpha}{\beta^2})\geq -\frac{\|x_0\|^2}{4}$, then
					\begin{align*}
						&\proj{\Calp}(x_0,-x_0,\gamma_0) =\\
						&\quad\mmenge{\Big(\dfrac{x_{0}}{2}+\dfrac{u}{\sqrt{2}}, -\dfrac{x_0}{2}+\dfrac{u}{\sqrt{2}}, \gamma_0-\dfrac{\alpha}{\beta^2} \Big)}{\|u\|=\sqrt{2\alpha\Big(\gamma_0-\frac{\alpha}{\beta^2}\Big)+\frac{\|x_0\|^2}{2}},~ u \in X}
					\end{align*}
					which is a singleton if and only if 
					$\alpha(\gamma_0-\frac{\alpha}{\beta^2})= -\frac{\|x_0\|^2}{4}$. 
					\item \label{t:2011411c} If  $y_0=x_0\neq 0$, 
					then we have the following: \\
					a) When $\alpha(\gamma_0+\frac{\alpha}{\beta^2})> \frac{\|x_0\|^2}{4}$,
					then 
					$$\proj{\Calp}(x_0,x_0,\gamma_0) = 
					\bigg\{\Big(\dfrac{x_{0}}{1+\lambda}, \dfrac{x_{0}}{1+\lambda}, \gamma_0+\frac{\lambda\alpha}{\beta^2} \Big) \bigg\}
					$$
					for a unique $\lambda\in \left]-1,1\right[$ that solves 
					the (essentially) cubic equation
					$$g_2(\lam) := \frac{2\|x_0\|^2}{(1+\lam)^2}- \frac{2\lam\alpha^2}{\beta^2}-2\alpha \gamma_0=0.
					$$
					b) When $\alpha(\gamma_0+\frac{\alpha}{\beta^2})\leq \frac{\|x_0\|^2}{4}$, then 
					\begin{align*}
						&\proj{\Calp}(x_0,x_0,\gamma_0)=\\ 
						&\quad \mmenge{\Big(\dfrac{x_{0}}{2}-\dfrac{v}{\sqrt{2}},\dfrac{x_0}{2}+\dfrac{v}{\sqrt{2}},\gamma_0+\dfrac{\alpha}{\beta^2}  \Big)}{
							\|v\|=\sqrt{-2\alpha\Big(\gamma_0+\frac{\alpha}{\beta^2}\Big)+\frac{\|x_0\|^2}{2}},~ v \in X}
					\end{align*}
					which is a singleton if and only if 
					$\alpha(\gamma_0+\frac{\alpha}{\beta^2})= \frac{\|x_0\|^2}{4}$. 
					\item \label{t:2011411d}
					If $x_0=y_0=0$, then we have the following: \\
					a) When $\alpha\gamma_0>\frac{\alpha^2}{\beta^2}$, 
					then the projection is the non-singleton set
					$$\proj{\Calp}(0,0,\gamma_0)= 
					\mmenge{\Big(\dfrac{u}{\sqrt{2}}, \dfrac{u}{\sqrt{2}}, \gamma_0-\frac{\alpha}{\beta^2}\Big)}{\|u\|=\sqrt{2\alpha\Big(\gamma_0-\frac{\alpha}{\beta^2}\Big)},\, u\in X}.
					$$
					b)   When $|\alpha\gamma_0|\leq\frac{\alpha^2}{\beta^2}$, 
					then 
					$$\proj{\Calp}(0,0,\gamma_0)=
					\big\{(0,0,0)\big\}. 
					$$
					c) When  $\alpha\gamma_0<-\frac{\alpha^2}{\beta^2}$, 
					then the projection is the non-singleton set
					$$\proj{\Calp}(0,0,\gamma_0)=
					\mmenge{\Big(-\dfrac{v}{\sqrt{2}}, \dfrac{v}{\sqrt{2}},\gamma_0+\dfrac{\alpha}{\beta^2} \Big)}{
						\|v\|=\sqrt{-2\alpha\Big(\gamma_0+\frac{\alpha}{\beta^2}\Big)},\, v\in X}.
					$$
				\end{enumerate}
			\end{theorem}
			\begin{proof}
				With
				$$A = \begin{bmatrix}
					\frac{1}{\sqrt{2}}\Id & -\frac{1}{\sqrt{2}}\Id & 0\\
					\frac{1}{\sqrt{2}}\Id & \frac{1}{\sqrt{2}}\Id & 0\\
					0 & 0 & 1
				\end{bmatrix}$$
				in mind,
				by \cref{p:transform}\cref{i:proj} we have
				\begin{align*}
					\proj{\Calp}[x_0,y_0,\gamma_0]^\intercal & =A\proj{\tCalp}A^{\intercal}[x_{0},y_{0},\gamma_{0}]^\intercal\\
					&=A\proj{\tCalp}
					\Big[\frac{x_0+y_0}{\sqrt{2}},\frac{-x_0+y_0}{\sqrt{2}},\gamma_0\Big]^\intercal.
				\end{align*}
				Hence \cref{t:2011411a}--\cref{t:2011411d} follow by applying \cref{p:2011d4.10}.
			\end{proof}
			
			\begin{remark} \cref{t:2011411}\cref{t:2011411a} was given in 
				\cite[Appendix~B]{elser2021learning}
				without a rigorous mathematical justification.
			\end{remark}
			
			It is interesting to ask what happens when $\alpha\rightarrow 0$.
			
			\begin{theorem} 
				Suppose that $X=\RR^n$. 
				Then $\proj{\Calp}\graphc \proj{C\times\RR}=\proj{C}\times\Id$ and
				$\proj{\tCalp}\graphc\proj{\Ct\times\RR}=\proj{\Ct}\times\Id$
				when $\alpha\rightarrow 0$.
			\end{theorem}
			\begin{proof} 
				Apply \cref{p:setconvergence} and \cref{f:projconvergence}.
			\end{proof}
			\begin{remark} The projection onto the cross $C$, $\proj{C}$, has been given in \cite{bauschke2022projection}.
			\end{remark}

			\section*{Acknowledgments}
			HHB and XW were supported by NSERC Discovery Grants. MKL was partially supported by a SERB-UBC Fellowship and NSERC Discovery Grants of HHB and XW.

			\bibliographystyle{siam}
			\bibliography{apao}

% \bib, bibdiv, biblist are defined by the amsrefs package.
\begin{bibdiv}
\begin{biblist}

\bib{aubin2009set}{book}{
      author={Aubin, Jean-Pierre},
      author={Frankowska, H{\'e}l{\`e}ne},
       title={Set-valued analysis},
   publisher={Springer Science \& Business Media},
        date={2009},
}

\bib{bauschke2022projection}{article}{
      author={Bauschke, Heinz~H},
      author={Lal, Manish~Krishan},
      author={Wang, Xianfu},
       title={The projection onto the cross},
        date={2022},
     journal={Set-Valued and Variational Analysis},
      volume={30},
      number={3},
       pages={997\ndash 1009},
        note={\url{https://doi.org/10.1007/s11228-022-00630-7}},
}

\bib{bauschke2023projections}{article}{
      author={Bauschke, Heinz~H},
      author={Lal, Manish~Krishan},
      author={Wang, Xianfu},
       title={Projections onto hyperbolas or bilinear constraint sets in
  {Hilbert} spaces},
        date={2023},
     journal={Journal of Global Optimization},
      volume={86},
      number={1},
       pages={25\ndash 36},
        note={\url{https://doi.org/10.1007/s10898-022-01247-8}},
}

\bib{bernard2005prox}{article}{
      author={Bernard, F.},
      author={Thibault, L.},
       title={Prox-regular functions in {Hilbert} spaces},
        date={2005},
     journal={Journal of Mathematical Analysis and Applications},
      volume={303},
      number={1},
       pages={1\ndash 14},
        note={\url{https://doi.org/10.1016/j.jmaa.2004.06.003}},
}

\bib{bertsekas1997nonlinear}{article}{
      author={Bertsekas, Dimitri~P},
       title={Nonlinear programming},
        date={1997},
     journal={Journal of the Operational Research Society},
      volume={48},
      number={3},
       pages={334\ndash 334},
}

\bib{elser2021learning}{article}{
      author={Elser, Veit},
       title={Learning without loss},
        date={2021},
     journal={Fixed Point Theory and Algorithms for Sciences and Engineering},
      volume={2021},
      number={1},
       pages={12},
        note={\url{https://doi.org/10.1186/s13663-021-00697-1}},
}

\bib{kreyszig1991introductory}{book}{
      author={Kreyszig, Erwin},
       title={Introductory functional analysis with applications},
   publisher={John Wiley \& Sons},
        date={1991},
      volume={17},
}

\bib{odehnal2020universe}{book}{
      author={Odehnal, Boris},
      author={Stachel, Hellmuth},
      author={Glaeser, Georg},
       title={The universe of quadrics},
   publisher={Springer Nature},
        date={2020},
}

\bib{rockafellar_variational_1998}{book}{
      author={Rockafellar, R.~Tyrrell},
      author={Wets, Roger J.~B.},
       title={Variational {Analysis}},
   publisher={Springer},
     address={Berlin},
      volume={317},
}

\end{biblist}
\end{bibdiv}

		\end{document}